\documentclass[10pt]{article}
\usepackage{amsmath}
\usepackage{amssymb}
\usepackage{amsfonts,amscd,amsthm}
\usepackage{fullpage,color}

\numberwithin{equation}{section}

\newtheorem{thm}{Theorem}
\newtheorem{prop}[thm]{Proposition}

\newtheorem{cor}[thm]{Corollary}

\newtheorem{defin}[thm]{Definition}

\newtheorem{lemma}[thm]{Lemma}
\newtheorem{example}[thm]{Example}

\newcommand{\GC}{G^\mathbb C}

\begin{document}

\title{Good Representations and Homogeneous Spaces}
\author{M. Jablonski}
\date{}
\maketitle

\textit{Caveat.}  After the first posting of this note on the arXiv, I was informed that proofs of these results have previously appeared in the literature.  The main theorem has appeared in \cite{Luna:ClosedOrbitsofReductiveGroups}, \cite{Nisnevich:StabilityAndIntersections}, and \cite{Vinberg:StabilityOfReductiveGroupActions}.  The proofs in \cite{Luna:ClosedOrbitsofReductiveGroups} and \cite{Vinberg:StabilityOfReductiveGroupActions} are more general than my own.  The proof in \cite{Nisnevich:StabilityAndIntersections} over $\mathbb C$ is similar to my proof.
I am grateful to Dmitri Panyushev and Vladimir L. Popov for bringing these works to my attention  and Rich\'ard Rim\'anyi for helping me to translate \cite{Nisnevich:StabilityAndIntersections}.\\

\textit{Acknowledgements}.  This note is part of my thesis work completed under the direction of Pat Eberlein at the University of North Carolina, Chapel Hill.  I would like to thank Shrawan Kumar for many worthwhile conversations.\\

\section{Main Results}
Let $G$ be a complex reductive affine algebraic group.  Let $F,H$ be algebraic reductive subgroups.  The homogeneous space $G/F$ has a natural, transitive left action of $G$ on it.  We will consider the induced action of $H$ on $G/F$.

Hereafter a property of a space will be called generic if it occurs on a nonempty Zariski open set.  Our main result is the following.\\

  The following theorem and its corollaries are true for both real and complex algebraic groups.  Moreover, the results actually hold for real semi-algebraic groups by passing to finite index subgroups and finite covers of manifolds.  We omit the details of the proofs for semi-algebraic groups.

\begin{thm}\label{thm: H action on G/F} Consider the induced action of $H$ on $G/F$, then generic $H$-orbits are closed in $G/F$; that is, there is a nonempty Zariski open set of $G/F$ such that the $H$-orbit of any point in this open set is closed.\end{thm}

\begin{cor}\label{cor: H normal} Let $G,H,F$ be as above.  If $H$ is normal in $G$, then all orbits of $H$ are closed in $G/F$.  Consequently, if $G$ acts on $V$ and the orbit $Gv$ is closed, then $Hv$ is also closed. \end{cor}

\begin{cor}\label{cor: intersection of red.} Let $G$ be a reductive algebraic group.  If $H,F$ are generic reductive subgroups, then $H\cap F$ is also reductive.  More precisely, take any two reductive subgroups $H$, $F$ of $G$.  Then $H\cap gFg^{-1}$ is reductive for generic $g\in G$.\end{cor}

\textit{Remark}.  The word generic cannot be replaced by all.  We show this in Example \ref{ex: non-red intersection}.  Theorem \ref{thm: H action on G/F} is proven first for complex groups then deduced for real groups.  It is not known at this time, to the author, whether or not these results hold true for more general algebraic groups.  Our proof exploits Weyl's Unitarian Trick.

Before proving this theorem, we present some corollaries to demonstrate its value.  Proofs of these results have been placed at the end.

\begin{cor}\label{cor: V G good implies V H good} Let $G$ be a reductive group acting linearly on $V$.  Let $H$ be a reductive subgroup of $G$.  If $G$ has generically closed orbits then $H$ does also.  Moreover, each closed $G$-orbit is stratified by $H$-orbits which are generically closed. \end{cor}

We say that a representation $V$ of $G$ is \textit{good} if generic $G$-orbits in $V$ are closed.

\begin{cor}\label{cor: V,W G good implies V+W is G good} Let $G$ be a reductive group, and let $V$ and $W$ be good $G$-representations, that is, generic $G$-orbits are closed.  Then $V\oplus W$ is also a good $G$-representation.\end{cor}

This corollary is of particular interest as it allows us to build good representations from smaller ones.  The idea of building good representations from subrepresentations was also carried out in \cite[Section 3]{EberleinJablo}.  In that setting, the representations of interest are those that have points whose real Mumford numerical function is negative.  The results of the current paper generalize some of those results.

\textit{Remark.} This corollary could have been stated more generally where $V,W$ are $G$-varieties.  Then the $G$-variety $V\times W$, with the diagonal action, has generically closed orbits.  The proof is the same.

\begin{example}[non-reductive intersection and non-closed orbit]\label{ex: non-red intersection} There exist semi-simple $G$ and reductive subgroups $H,F$ such that $H\cap F$ is not reductive.  Additionally, we demonstrate a representation $V$ of $G$ so that $G \cdot x$ is closed but $H\cdot x$ is not closed, for some $x\in V$.
\end{example}

Recall the following well-known fact.  Let $G$ be an reductive affine algebraic group acting on an affine variety.  If the orbit $G\cdot x$ is closed then $G_x$ is reductive, see \cite[Theorem 3.5]{BHC} or \cite[Theorem 4.3]{RichSlow}.  We will choose $F=G_x$ for a particular $x\in V$.  Then $H\cap F = H\cap G_x = H_x$.  Once it is shown that $H_x$ is not reductive, the orbit $H\cdot x$ cannot be closed by the fact stated above.

Consider $G=SL_6(\mathbb C)$ acting on $V=\bigwedge ^2 \mathbb C^6 \simeq \mathfrak{so}(6,\mathbb C)$.  This is the usual action and is described as follows.  For $M\in \mathfrak{so}(6,\mathbb C)$ and $g\in SL_6 (\mathbb C)$,  the action is defined as $g\cdot M = gMg^t$.  The subgroup $H = SL_2(\mathbb C)$ is imbedded as the upper left $2\times 2$ block.

Let $v\in V$ be the block diagonal matrix consisting of the blocks $J=\begin{bmatrix} 0 & 1\\ -1& 0 \end{bmatrix}$ along the diagonal.  That is, $v=\begin{bmatrix} J\\ & J\\ && J\end{bmatrix}$.  Given the standard inner product (from the trace form) on $V$, the vector $v$  is a so-called minimal vector as $v^2= - Id$, thus $G\cdot v$ is closed.  See \cite[Example 1]{EberleinJablo} for details and more information on minimal vectors; see also \cite{RichSlow}.  Consider $x=g\cdot v$ where
$$g=\begin{bmatrix} 1&0&1\\&1&0\\&&1\\ \\&&&& Id_{3\times 3}\\ &&&&&\end{bmatrix}$$
Since $G_v$ is reductive, $G_x = gG_v G^{-1}$ is also reductive.  One can  compute $H_x$ to show that $H_x \simeq \mathbb C_a = \begin{bmatrix}1&a\\0&1\end{bmatrix}$.  This group is clearly not reductive and we have the desired example.\\

\section{Technical Lemmas}

We recall the definition of varieties and morphisms which are defined over $\mathbb R$.  This is the setting that we will primarily work in.  See \cite[$\S \S 11-14$]{Borel:LinAlgGrps} or \cite[Chapter 1, 0.10]{Margulis} for more information on varieties and $k$-structures on varieties.

\begin{defin}[Real points of affine subvarieties] An affine subvariety $M$ of $\mathbb C^n$ is the zero set of a collection of polynomials on $\mathbb C^n$.  The variety $M$ is said to be defined over $\mathbb R$  if $M$ is the zero set of polynomials whose coefficients are real.  Thus $\mathbb C[M]=\mathbb R[M]\otimes_\mathbb R \mathbb C$.  The real points of $M$ are defined as the set $M(\mathbb R)=M\cap \mathbb R^n$; we call such a set a real variety.\end{defin}

\begin{defin}[$\mathbb R$-structures] Given an abstract affine variety $X$, one defines a $\mathbb R$-structure on $X$ by means of an isomorphism $\alpha :X\to M$.  A morphism $f:X\to Y$ of $\mathbb R$-varieties is said to be defined over $\mathbb R$ if the comorphism $f^*:\mathbb C [Y] \to \mathbb C [X]$ satisfies  $f^*(\mathbb R [Y]) \subset \mathbb R [X]$.  Additionally, we define the real points of $X$ to be $X(\mathbb R)=\alpha ^{-1}(M(\mathbb R))$. \end{defin}

\textit{Remark}.  Let $M \subset \mathbb C^m$ and $N\subset \mathbb C^n$ be subvarieties defined over $\mathbb R$.  Then $f:M\to N$ being defined over $\mathbb R$ implies $f(M(\mathbb R))\subset N(\mathbb R)$.  To obtain the converse one needs $M$ to have an additional property that we call the (RC) property, see Definition \ref{def: RC property}.  We state the converse after defining this property.

We observe that a variety can be endowed with many different real structures.

\begin{defin}\label{def: RC property}[RC - property] Let $X$ be a complex variety defined over $\mathbb R$.  We say that $X$ has the (RC) property (real-complexified) if the real points $X(\mathbb R)$ are Zariski dense in $X$.\end{defin}

This scenario arises precisely if one begins with a real variety $Z\subset \mathbb R^n$ and considers the Zariski closure $\bar Z\subset \mathbb C^n$.  Here $\bar Z$ has the (RC) property; see \cite{Whitney} for an introduction to real varieties and their complexifications.

\begin{prop}\label{prop: f:M to N defined over R} Let $M \subset \mathbb C^m$ and $N\subset \mathbb C^n$ be subvarieties defined over $\mathbb R$.  Assume that $M$ has the (RC) property.  Then $f:M\to N$ being defined over $\mathbb R$ is equivalent to $f(M(\mathbb R))\subset N(\mathbb R)$.\end{prop}

This result is useful but not needed in our proofs; we postpone the proof of this proposition till the end.  One immediately sees that the same holds more generally for abstract affine varieties with $\mathbb R$-structures.  That is, let $X,Y$ be complex affine varieties defined over $\mathbb R$ and $f:X\to Y$ a morphism.  Assume that $X$ has the (RC) property.  Then $f$ is defined over $\mathbb R$ if and only if $f(X(\mathbb R))\subset Y(\mathbb R)$.  Often we will simply say that $f:X\to Y$ is defined over $\mathbb R$, or $f$ is an $\mathbb R$-morphism, when both varieties and the morphism are defined over $\mathbb R$.

\begin{lemma}\label{lemma: BHC abstractly} Let $X$ be a complex affine variety and $G$ a complex reductive affine algebraic  group acting on $X$, with all defined over $\mathbb R$.  Then $G(\mathbb R)$ acts on $X(\mathbb R)$ and for $x\in X(\mathbb R)$ the orbit $G(\mathbb R)\cdot x$ is Hausdorff closed in $X(\mathbb R)$ if and only if $G\cdot x$ is Zariski closed in $X$.\end{lemma}

\textit{Remark}. It is well-known that $G\cdot x$ is Hausdorff closed if and only if it is Zariski closed, see \cite{Borel:LinAlgGrps}.  Notice that the above situation arises when we have a real algebraic group acting on a real algebraic variety.  This lemma has been proven for linear $G$ actions, see \cite[Proposition 2.3]{BHC} and \cite{RichSlow}; we reduce to this case.

\begin{proof} Let $G$ and $X$ be as above.  It is well-known that there exists a complex vector space $V$ (defined over $\mathbb R$), a closed $\mathbb R$-imbedding $i:X \hookrightarrow V$, and a representation $T: G\to GL(V)$ defined over $\mathbb R$ such that $i(gx)=T(g)i(x)$ for all $g\in G , x\in X$.  See \cite[I.1.12]{Borel:LinAlgGrps} for the construction of such an imbedding.

As all of our objects are defined over $\mathbb R$ we see that $G(\mathbb R)$ acts on $X(\mathbb R)$, $i(X(\mathbb R))\subset V(\mathbb R)$, and $T: G(\mathbb R) \to GL(V(\mathbb R)$ is a real linear representation of $G(\mathbb R)$, cf. the remark before Definition \ref{def: RC property}.

Now take $x\in X(\mathbb R)$.  We have the following set of equivalences
$$ \begin{array}{rl} G(\mathbb R)x & \mbox{is closed in }X(\mathbb R)\\
    i(G(\mathbb R)x) & \mbox{is closed in $V(\mathbb R)$, as $i$ is a closed $\mathbb R$-imbedding}\\
    T(G(\mathbb R))i(x) &\mbox{is closed in }V(\mathbb R)\\
    T(G)i(x) &\mbox{is closed in }V \mbox{ by \cite{BHC, RichSlow}}\\
    i(Gx)&\mbox{is closed in }V\\
    Gx&\mbox{is closed in }X
\end{array}$$ \end{proof}

\subsection*{Quotients}

Recall that if $F$ is a complex reductive affine algebraic group acting on a complex affine variety $X$, there exists a ``good quotient'' $X//F$ from Geometric Invariant Theory (GIT).  Here $X//F$ is an affine variety together with a quotient morphism $\pi : X\to X//F$ which is a regular map between varieties.  The variety $X//F$ has as its ring of regular functions $\mathbb C[X//F] = \mathbb C[X]^F$, the $F$-invariant polynomials on $X$.  Moreover, the quotient map is the morphism corresponding to the injection $\mathbb C[X]^F \hookrightarrow \mathbb C[X]$.  See \cite{Newstead} for a detailed introduction to Geometric Invariant Theory and quotients.

Good quotients are categorical quotients (see \cite[Chapter 3]{Newstead}).  As a consequence they possess the following universal property which will be needed later.  Let $\phi : X \to Z$ be a morphism which is constant on $F$-orbits.  Then there exists a unique morphism $\varphi : X//F \to Z$ such that $\varphi \circ \pi = \phi$.

Our application of GIT quotients is the following.  Let $G$ be a reductive group and $F$ a reductive subgroup.  The group $F$ acts on $G$ via $f\cdot g = gf^{-1}$.  This gives a left action of $F$ on $G$ such that every orbit is closed.  In this way the GIT quotient $G//F$ is a parameter space; that is, every $F$-orbit is closed.  If one considers the analytic topologies on $G$ and $G//F$ one readily sees that $G//F$ and $G/F$ (with the usual Hausdorff quotient topology) are homeomorphic.  In this way we endow $G/F$ with a Zariski topology.  Here and in later discussion we identify the coset space $G/F$ with the variety $G//F$.  Moreover, it will be shown that the natural $G$-action on $G/F$ is algebraic.

Richardson and Slodowy \cite{RichSlow} have shown the following
\begin{prop}\label{prop: quotient is defined over R} Let $G$ be a reductive algebraic group acting on $X$ so that $G$, $X$, and the action are defined over $\mathbb R$ and consider the quotient morphism $\pi :X\to X//G$.  Then $\pi$ is defined over $\mathbb R$ and $\pi (X(\mathbb R))\subset (X//G)(\mathbb R)$ is Hausdorff closed. \end{prop}

In general one cannot expect $\pi (X(\mathbb R))$ to be all the real points $(X//G) (\mathbb R)$.  However, we make the following simple observation.

\begin{lemma}\label{lemma: RC prop of quotient} If $X$ has the (RC) property, then so does $X//G$.  In fact $\pi (X(\mathbb R))$ is Zariski dense in $X//G$.\end{lemma}

The first statement is proven in \cite{RichSlow} and the second statement is a special case of a more general statement: Let $f:X\to Y$ be a regular map and $Z$ a Zariski dense set of $X$, then $f(Z)$ is Zariski dense in $f(X)$.

\begin{prop}\label{prop: H on G/F defined over R} Let $G$ be a reductive algebraic group defined over $\mathbb R$ and $H,F$ algebraic reductive subgroups defined over $\mathbb R$.  Then the action of $H$ on $G/F$ is defined over $\mathbb R$.\end{prop}

Before presenting the proof of this proposition  we state the following useful lemma.

\begin{lemma}\label{lemma: H x F on X over R} Let $H\times F$ act on a variety $X$, where $H$, $F$, $X$, and the actions are defined over $\mathbb R$.  Then there is a unique
$H$ action on $X//F$ defined over $\mathbb R$ which makes Diagram A (below) commute.\end{lemma}

\begin{proof}[Proof of lemma] Since $H\times F$ acts on X we can consider the $F$ action on $H\times X$.  We claim that the map $\pi_1=id\times \pi_2: H\times X \to H\times (X//F)$ is a good quotient; where $\pi_2: X\to X//F$ is a good quotient.  Here $X//F$ is the variety whose ring of regular functions is $\mathbb C[X]^F$, the $F$-invariant polynomials of $\mathbb C [X]$.   For a detailed introduction to quotients, see \cite[Chapter 3]{Newstead}.

To show that $H\times (X//F)$ is the desired quotient, we will show that $\mathbb C[H\times (X//F)] = \mathbb C[H\times X]^F$ and that the comorphism $(id\times \pi_2)^*$ is the inclusion map.  Recall that $\pi_2^* : \mathbb C[X//F]=\mathbb C[X]^F \hookrightarrow \mathbb C [X]$ is the inclusion map.

There is a natural identification between $\mathbb C[H\times X]$ and $\mathbb C[H]\otimes \mathbb C[X]$ defined by $\sum p_i(h)q_i(x) \mapsto (\sum p_i\otimes q_i) (h,x)$.  Under this identification, $\mathbb C[H\times X]^F\simeq \mathbb C[H] \otimes \mathbb C[X]^F$ where $\mathbb C[H\times X]^F, \mathbb C[X]^F$ denote the $F$-invariant polynomials in $\mathbb C[H\times X], \mathbb C[X]$, respectively.  The map $id\times \pi_2: H\times X \to H\times (X//F)$ corresponds to a comorphism $(id\times \pi_2)^*: \mathbb C[H\times (X//F)] \to \mathbb C[H\times X]$ and under the natural identification described above, the map $(id\times \pi_2)^*$ corresponds to $id^* \otimes \pi_2^* : \mathbb C[H]\otimes \mathbb C[X//F] \to \mathbb C[H]\otimes \mathbb C[X]$.  This map $id^*\otimes \pi_2^*$ is the inclusion map and is an isomorphism onto $\mathbb C[H]\otimes \pi_2^*(\mathbb C[X//F])=\mathbb C[H]\otimes \mathbb C[X]^F$.  Thus, $(id\times \pi_2)^*$ is the inclusion map and  maps $\mathbb C[H\times (X//F)]$ isomorphically onto $\mathbb C[H\times X]^F$.  We have shown the following.
$$ H\times (X//F) \simeq (H\times X)//F$$

Consider the following diagram.  Let $m_1$ denote the morphism corresponding to $H$-action on $X$. Since $\pi_2\circ m_1 $ is constant on $F$-orbits, by the discussion above concerning quotients,  there exists a unique map $m_2$ which factors and makes the diagram commute.
\[ \begin{CD}
H\times X   @>m_1>>    X \\
@V\pi_1VV                               @VV\pi_2V \\
H\times (X//F)          @>m_2>>       X//F
\end{CD} \tag{A} \]
where $\pi_1=id\times \pi_2$ is the quotient of the $F$ action on $H\times X$, $f\cdot(h,x)=(h,f\cdot x)$.    Equivalently, for $h\in H$ and a closed orbit $F\cdot x \subset X$, $h(Fx)=F(hx)$ is a closed $F$-orbit.

We know that $m_1, \pi_1, \pi_2$ are defined over $\mathbb R$ and that $m_2 \circ \pi_1=\pi_2 \circ m_1$ is defined over $\mathbb R$.  From this we wish to show $m_2$ is also defined over $\mathbb R$.  Since $\pi_2^* (\mathbb R[X//F]) = \mathbb R [X]^F$ we have
$$ \pi_1^* \circ m_2^* = m_1^* \circ \pi_2^* : \mathbb R [X//F] \to \mathbb R[H\times X]^F$$

Since $\pi_1^* : \mathbb C[H\times (X//F)] \to \mathbb C[H\times X]^F$ and $\pi_1^* : \mathbb R[H\times (X//F)] \to \mathbb R[H\times X]^F$ are isomorphisms, we have
    $$ m_2^*(\mathbb R[X//F])\subset \pi_1^{*\ -1}(\mathbb R[H\times  X]^F)=\mathbb R[H\times X//F]$$

Thus,  $m_2:H\times (X//F) \to X//F$ is defined over $\mathbb R$, or equivalently, the $H$ action on $X//F$ defined by $m_2$ is defined over $\mathbb R$.  The uniqueness of the $H$-action on $X//F$ is equivalent to the uniqueness of the map $m_2$ in Diagram A.
\end{proof}

\begin{proof}[Proof of the proposition]

Once it is shown that the $G$-action on $G/F$ is defined over $\mathbb R$, it will be clear that the $H$-action is also defined over $\mathbb R$.  We apply Lemma \ref{lemma: H x F on X over R} in the setting that $G$ is a reductive group, $H=G$, $X=G$, and $F$ is a reductive subgroup of $G$.

Since $G$ is an algebraic group defined over $\mathbb R$, the action $G\times F$ on $G$ defined by $(h,f)\cdot g=hgf^{-1}$ is defined over $\mathbb R$, where $h,g\in G$, $f\in F$.  Recall that  $G/F$ is  the GIT quotient $G//F$ under the $F$ action listed above (notice all the orbits are closed, hence the usual topological quotient coincides with the algebraic quotient).  The unique $H$-action described in Lemma \ref{lemma: H x F on X over R} is precisely the standard action of $G$ on $G/F$.  Thus we have shown that the usual action of $G$ on $G/F$ is algebraic and defined over $\mathbb R$.

\end{proof}

%===============================

\section{Transitioning between the real and complex settings: Proof of Theorem 1}

First we remark on how one obtains  Theorem \ref{thm: H action on G/F} for real algebraic groups once it is known for complex groups.  Let $G,H,F$ be the same as in  Theorem \ref{thm: H action on G/F} but with real groups instead of complex groups. Let $G^\mathbb C$ denote the
Zariski closure of $G$ in $GL(n,\mathbb C)$; this variety is a complex algebraic subgroup of $GL(n,\mathbb C)$.  It follows that $G$ is the set of real points of $G^\mathbb C$,  and we call $\GC$ the \textit{algebraic complexification} of $G$.
Likewise, $H,F$ are the real points of their complexifications $H^\mathbb C,F^\mathbb C$.  Here all of our objects have the (RC)-property.

Consider the $G$-equivariant imbedding $i:G/F \to \GC / F^\mathbb C$, defined by $i: gF \mapsto gF^\mathbb C$, and the quotient $\pi : \GC \to \GC /F^\mathbb C$.  Note that $i$ is injective since $G\cap F^\mathbb C = F$.  We view $G/F$ as a subset of $\GC / F^\mathbb C$ via $i$ and we note that $i(G/F)=\pi (G)$.

As $G/F \simeq \pi(G)$ and $G=G^\mathbb C(\mathbb R)$, we see that $G/F \subset (G^\mathbb C/F^\mathbb C)(\mathbb R)$  and is Zariski dense in $G^\mathbb C/F^\mathbb C$ (see Proposition \ref{prop: quotient is defined over R} and Lemma \ref{lemma: RC prop of quotient}).  Moreover, assuming the theorem is true in the complex setting, there exists a Zariski open set $\mathcal O \subset G^\mathbb C/F^\mathbb C$ such every point in $\mathcal O$ has a closed $H^\mathbb C$ orbit.  $G/F$ being Zariski dense intersects $\mathcal O$ and so, by Lemma \ref{lemma: BHC abstractly}, we see that all points of $G/F \cap \mathcal O$ have closed $H$-orbits in $G/F$.  This proves Theorem \ref{thm: H action on G/F} in the real case.\\

To prove the theorem for complex groups, we take advantage of certain real group actions.  Let $G$ be a complex reductive group and $U$ a maximal compact subgroup.  We can realize $U$ as the fixed points of a Cartan involution $\theta$.  Moreover, there exists a real structure on $G$ so that $U$ is the set of real points of $G$ (see \cite[Remark 3.4]{BHC}). Observe that $G$ has the (RC) property as $G$ is the complexification of its compact real form $U$.  We state this below.

\begin{lemma}[Weyl's Unitarian Trick] Let $G$ be a complex reductive group and $U$ a maximal compact subgroup.  Then $U$ is Zariski dense in $G$. \end{lemma}

\begin{lemma} We may assume that $H,F$ from our main theorem are $\theta$-stable.\end{lemma}

\begin{proof}  It is well-known that there exist conjugations $g_1 H g_1^{-1}$ and $g_2Fg_2^{-1}$ so that these conjugates are $\theta$-stable, see \cite{BHC} or \cite{Mostow:SelfAdjointGroups}.  So to prove the lemma, we just need to show that the theorem holds for $H,F$ if and only if it holds for conjugates of these groups.

Observe that $G/F$ and $G/(g_2Fg_2^{-1})$ are isomorphic as varieties via conjugation by $g_2 : gF \mapsto (g_2 g g_2^{-1})(g_2 F g_2^{-1})$.  We denote this map by $C(g_2)$.  Also observe that $G$ acts via left translation on $G/F$ by variety isomorphisms.  Thus the left translate of a closed set in $G/F$ is again a closed set  $G/F$.  For $k\in G$ we have  $(g_1Hg_1^{-1})k(g_2Fg_2^{-1})$  is closed in $G/(g_2Fg_2^{-1})$ if and only if $C(g_2)^{-1}\cdot ((g_1Hg_1^{-1})k(g_2Fg_2^{-1}))$ is closed in $G/F$.  But $C(g_2)^{-1} \cdot ((g_1Hg_1^{-1})k(g_2Fg_2^{-1}))  = g_2^{-1}g_1Hg_1^{-1}kg_2 F$ and this is closed if and only if $ Hg_1^{-1}kg_2 F$ is closed in $G/F$.

Thus the  $g_1Hg_1^{-1}$-orbit of $k (g_2Fg_2^{-1})$ is closed in $G/(g_2 F g_2^{-1})$ if and only if the $H$-orbit of $g_1^{-1}kg_2 F$ is closed in $G/F$.

\end{proof}

We continue the proof of Theorem \ref{thm: H action on G/F}.
Now that $H,F$ are $\theta$-stable, and $U=Fix(\theta)$, we know that their maximal compact subgroups $U_H = U\cap H,U_F=U\cap F$ are contained in $U$.  Moreover, since $U=G(\mathbb R)$, the compact subgroups $U_H,U_F$ are the real points of the algebraic groups $H,F$.  Observe that $H,F$ have the (RC) property as their maximal compact subgroups are the real points.  For a proof of the following useful fact in the complex setting see \cite{Newstead}.  For an extension to the real setting see Section 2 of \cite{Jablo:Thesis}.

\begin{prop}\label{prop: Newstead}Let $G$ be a real or complex linear reductive algebraic group acting on an affine variety $X$.  If there exists a closed orbit of maximal dimension, then there is a Zariski open set of such orbits. \end{prop}

\begin{proof}[Proof of Theorem \ref{thm: H action on G/F}]  We apply the above proposition to the action of $H$ on the affine variety $G/F$.  Note that $G/F$ is affine as $G$ is reductive and $F$ is reductive (see \cite[Theorem 3.5]{BHC}).

As the  $F$-action on $G$ is defined over $\mathbb R$, the quotient $G/F$ is defined over $\mathbb R$.  Since our objects have property (RC) the image of the real points of $G$ is dense in $G/F$ by Lemma \ref{lemma: RC prop of quotient}; that is, $U/U_F \subset G/F$ is dense.  Here, as before, we are identifying $U/U_F$ with the image of $U$ under the quotient $G\to G/F$.

Moreover, Proposition \ref{prop: H on G/F defined over R}  shows that the $H$-action on $G/F$ is defined over $\mathbb R$.  If we let $\mathcal O$ denote the set of maximal dimension $H$-orbits in $G/F$, then $\mathcal O \cap (U/U_F)$ is non-empty.  However, $U_H$ is the set of real points for $H$ and every $U_H$ orbit in $U/U_F$ is closed (since they are all compact).  Therefore, every point in $\mathcal O \cap (U/U_F)$ has a closed $H$-orbit by Lemma \ref{lemma: BHC abstractly} and we have found a closed $H$-orbit of maximal dimension.

Applying Proposition \ref{prop: Newstead} we see that generic $H$-orbits are closed.
\end{proof}

%=======================================

\section{Proofs of Corollaries}

We note that the proofs of all results below, except for Proposition \ref{prop: f:M to N defined over R}, are valid in the real and complex cases simultaneously.

\begin{proof}[Proof Corollary \ref{cor: H normal}]  Theorem \ref{thm: H action on G/F} provides some point $kF \in G/F$ which has a closed $H$-orbit.  Take $g\in G$ and consider the point $gk F \in G/F$.  The $H$-orbit of this point is $(H\cdot gk) F = (g \ g^{-1}Hgk )F = (g\ Hk )F$ which is closed as the $G$ action on $G/F$ is by variety isomorphisms.  Hence every $H$-orbit in $G/F$ is closed as $G$ acts transitively on $G/F$.

We prove the second statement of the corollary using Corollary \ref{cor: V G good implies V H good} (which is proven below).  In the proof of this  corollary it is shown that if $G\cdot v$ is closed in $V$, then there exists $g\in G$ such that $H gv$ is closed in $V$.  But now $Hgv = gHv$ by the normality of $H$.  Moreover, $gHv$ is closed in $V$ if and only if $Hv$ is closed in $V$ as $G$ acts by isomorphisms of the vector space.  This proves the second part of the proposition.
\end{proof}

\begin{proof}[Proof of Corollary \ref{cor: V G good implies V H good}]  We prove the second statement first.  Take $v\in V$ such that $G\cdot v$ is closed. It is well-known that $G_v$ is reductive, see, e.g., \cite[Theorem 4.3]{RichSlow} or \cite[Theorem 3.5]{BHC}.  The orbit $G\cdot v$ is $G$-equivariantly isomorphic to the affine variety $G/G_v$.  Thus the $H$-orbit $H\cdot gv \subset G\cdot v \subset V$ corresponds to $H\cdot g G_v \subset G/G_v$ and for generic $g$ these $H$-orbits are closed by Theorem \ref{thm: H action on G/F}.  This proves the second statement.

For the first statement, let $\mathcal O =\{\ v\in V\ | \dim \ H\cdot v\mbox{ is maximal} \}$ and let $\mathcal U =\{ \ v\in V\ |\ G\cdot v \mbox{ is closed }\}$.  The set $\mathcal O$ is a nonempty Zariski open set and by hypothesis $\mathcal U$ contains a nonempty Zariski open set.  Pick $w\in \mathcal O \cap \mathcal U$.  For generic $g\in G$,  the orbit $H\cdot gw$ is closed by the argument of the previous paragraph. Moreover, $gw \in \mathcal O \cap \mathcal U$ for generic $g\in G$.  Thus there exists some point $gw$ which has a closed $H$-orbit of maximal dimension.  Therefore by Proposition \ref{prop: Newstead} generic $H$-orbits in $V$ are closed.
\end{proof}

\begin{proof}[Proof of Corollary \ref{cor: V,W G good implies V+W is G good}]  Take $v\in V$ and $w\in W$ which both have closed $G$-orbits.  Then the $G\times G$ orbit of $(v,w)$ is closed in $V\oplus W$.  Now consider the diagonal imbedding of $G$ in $G\times G$.  In this way, $G$ acts on $V\oplus W$ and since generic $G\times G$-orbits in $V\oplus W$ are closed, we see that generic $G$-orbits in $V\oplus W$ are also closed by Corollary \ref{cor: V G good implies V H good}. \end{proof}

\begin{proof}[Proof of Corollary \ref{cor: intersection of red.}]  Let $G$ be a reductive group and let $H$, $F$ be reductive subgroups.  There exists a representation $V$ of $G$ such that the  reductive subgroup $F$ can be realized as the stabilizer of a point $v\in V$ and such that the orbit $G\cdot v$ is closed, see \cite[Proposition 2.4]{BHC}.

By Corollary \ref{cor: V G good implies V H good}, we know that $H\cdot gv$ is closed for generic $g\in G$.  Thus $H_{gv}$ is reductive for generic $g\in G$.  But $H_{gv}=H\cap G_{gv}=H\cap gG_vg^{-1}=H\cap gFg^{-1}$ and we have the desired result.
\end{proof}

\begin{proof}[Proof of Proposition \ref{prop: f:M to N defined over R}]
First we remark on the direction that does not require $M$ to have the (RC) property; that is, if $f:M\to N$ is defined over $\mathbb R$ then $f(M(\mathbb R)) \subset N(\mathbb R)$.  To see this direction write $f=(f_1,\dots , f_n)$, where $f_i :\mathbb C^m \to \mathbb C$.  These component functions are precisely $f_i=f^*(\pi_i)$ where $\pi_i$ is projection from $\mathbb C^n$ to the $i$-th coordinate.  Since this projection is defined over $\mathbb R$ we see that the $f_i$ take real values when evaluated at real points.  That is, $f ( M\cap \mathbb R ^m) \subset N\cap \mathbb R^n$.

Now assume $M$ has the (RC) property and let $f:M\to N$ be a morphism of varieties such that $f (M(\mathbb R)) \subset N(\mathbb R)$.  We will show $f^*(\mathbb R [N]) \subset \mathbb R [M]$; that is, $f$ is defined over $\mathbb R$.

We can describe the polynomial $f$ by its coordinate functions, $f=(f_1,\dots,f_n)$ where $f_i:\mathbb C^m \to \mathbb C$ and $f_i|_{M\cap \mathbb R^m} \to \mathbb R$.  Let $\overline f_i$ denote the polynomial whose coefficients   are the complex conjugates of those of $f_i$, then we have $\frac{1}{2}(f_i + \overline f_i )=f_i$ on the set $M\cap \mathbb R^m$.  $M$ having the (RC) property means precisely that $M\cap \mathbb R^m$ is Zariski dense in $M$, thus $\frac{1}{2}(f_i + \overline f_i )=f_i$ on $M$.  If we define $P = \frac{1}{2}(f+\overline f)$ then $P$ has real coefficients and restricted to $M$ equals $f$.

Take $g\in \mathbb R[N]$, then $f^*(g) \in \mathbb C[M]$ and $f^*(g)=g\circ f = g \circ P$ on $M$.  Since $g$ and $P$ have real coefficients, so does their composition.  That is, $f^*(g) \in \mathbb R[M]$.
\end{proof}


\begin{thebibliography}{BHC62}

\bibitem[BHC62]{BHC}
Armand Borel and Harish-Chandra, \emph{Arithmetic subgroups of algebraic
  groups}, The Annals of Mathematics \textbf{75} (1962), no.~3, 485--535, 2nd
  ser.

\bibitem[Bor91]{Borel:LinAlgGrps}
Armand Borel, \emph{Linear algebraic groups}, second ed., Graduate Texts in
  Mathematics, vol. 126, Springer-Verlag, 1991.

\bibitem[EJ]{EberleinJablo}
Patrick Eberlein and Michael Jablonski, \emph{Closed orbits and the real {M}
  function}, in preparation.

\bibitem[Jab08]{Jablo:Thesis}
Michael Jablonski, \emph{Real geometric invariant theory and ricci soliton
  metrics on two-step nilmanifolds}, Thesis (May 2008).

\bibitem[Lun72]{Luna:ClosedOrbitsofReductiveGroups}
Domingo Luna, \emph{Sur les orbites ferm\'ees des groupes alg\'ebriques
  r\'eductifs}, Invent. Math. \textbf{16} (1972), 1--5.

\bibitem[Mar91]{Margulis}
G.A. Margulis, \emph{Discrete subgroups of semisimple {L}ie groups}, Ergebnisse
  der Mathematik und ihrer Grenzgebiete, 3. Folge, Springer-Verlag, New York
  Berlin Heidelberg, 1991.

\bibitem[Mos55]{Mostow:SelfAdjointGroups}
G.~D. Mostow, \emph{Self-adjoint groups}, Ann. of Math. (2) \textbf{62} (1955),
  44--55.

\bibitem[New78]{Newstead}
P.E. Newstead, \emph{Lectures on introduction to moduli problems and orbit
  spaces}, Springer-Verlag, published for the Tata Institute of Fundamental
  Research, Bombay, Berlin New York, 1978.

\bibitem[Nis73]{Nisnevich:StabilityAndIntersections}
E.~A. Nisnevi{\v{c}}, \emph{Intersection of the subgroups of a reductive group,
  and stability of the action}, Dokl. Akad. Nauk BSSR \textbf{17} (1973),
  785--787, 871(in Russian).

\bibitem[RS90]{RichSlow}
R.W. Richardson and P.J. Slodowy, \emph{Minimum vectors for real reductive
  algebraic groups}, J. London Math. Soc. \textbf{42} (1990), 409--429.

\bibitem[Vin00]{Vinberg:StabilityOfReductiveGroupActions}
E.B. Vinberg, \emph{On stability of actions of reductive algebraic groups}, Lie
  Algebras, Rings and Related Topics, Papers of the 2nd Tainan-Moscow
  International Algebra Workshop'97, Tainan, Taiwan, January 11-17, 1997, (Hong
  Kong) (et~al. Fong, Yuen, ed.), Springer, 2000, pp.~188--202.

\bibitem[Whi57]{Whitney}
Hassler Whitney, \emph{Elementary structure of real algebraic varieties}, The
  Annals of Mathematics \textbf{66} (1957), no.~3, 545--556, 2nd Ser.

\end{thebibliography}
\end{document}